\documentclass[12pt]{amsart}
\usepackage[margin=3.7cm]{geometry}
\usepackage{amssymb}
\usepackage[all]{xy}

\title{Noncommutative resolution, F-blowups and $D$-modules}

\author[Y. Toda]{Yukinobu Toda}
\address{Institute for the Physics and Mathematics of the Universe (IPMU), University of Tokyo, 
Kashiwano-ha 5-1-5, Kashiwa City, Chiba 277-8582, Japan}
\email{toda@ms.u-tokyo.ac.jp, toda-914@pj9.so-net.ne.jp}

\author[T. Yasuda]{Takehiko Yasuda}
\address{Department of Mathematics and Computer Science, 
Faculty of Science, Kagoshima University, 1-21-35 Korimoto, Kagoshima 890-0065, Japan}
\email{yasuda@sci.kagoshima-u.ac.jp}

\theoremstyle{plain}
\newtheorem{thm}{Theorem}[section]

\newtheorem{prop}[thm]{Proposition}
\newtheorem{cor}[thm]{Corollary}
\newtheorem{lem}[thm]{Lemma}

 \theoremstyle{definition}
\newtheorem{defn}[thm]{Definition}

\newtheorem{prob}[thm]{Problem}

\theoremstyle{remark}
\newtheorem{rem}[thm]{Remark}

\def\AA{\mathbb A}

\def\ZZ{\mathbb Z}

\def\Zp{\mathbb{Z}_{>0}}
\def\Znn{\mathbb{Z}_{\ge 0}}

\def\bD{\mathbf{D}}

\def\fm{\mathfrak m}

\def\11{\mathbf{1}}

\def\lmodules{\text{-}\mathbf{mod}}

\DeclareMathOperator{\Spec}{Spec}

\DeclareMathOperator{\Hom}{Hom}

\DeclareMathOperator{\FB}{FB}

\DeclareMathOperator{\End}{End}

\DeclareMathOperator{\dimbf}{\mathbf{dim}}

\DeclareMathOperator{\gldim}{gldim}

\def\GHilb{\mathrm{Hilb}^{G}}
\def\GCoh{\mathbf{Coh}^{G}}
\def\Coh{\mathbf{Coh}}

\begin{document}

\maketitle

\begin{abstract}
We explain the isomorphism between the $G$-Hilbert scheme and the F-blowup
from the noncommutative viewpoint after Van den Bergh.
In doing this, we immediately and naturally arrive at the notion of $D$-modules. 
We also find, as a byproduct, a canonical way to construct a
noncommutative resolution at least for a few classes of singularities 
in positive characteristic.
\end{abstract}

\section{Introduction}

The starting point of this work is the isomorphism between the $G$-Hilbert
scheme and the F-blowup found in \cite{Yasuda:math0706.2700}.
The $G$-Hilbert scheme, introduced by Ito and Nakamura \cite{MR1420598},
is associated to a smooth  $G$-variety $M$ with $G$ a finite group,
while the $e$-th F-blowup introduced by the second author \cite{Yasuda:math0706.2700} is associated to 
 the $e$ times iteration of the  Frobenius morphism $F: X \to X$
 of a singular variety in positive characteristic. 
 Both are the moduli spaces of certain $0$-dimensional subschemes.
 
 From now on, we work over  an algebraically closed field $k$ of
 characteristic $p >0$. 
Then under some condition,
 the mentioned isomorphism connects the $G$-Hilbert scheme, $\GHilb(M)$, 
of a $G$-variety $M$ and 
the $e$-th F-blowup, $\FB_{e}(X)$, of the quotient variety $X=M/G$.
\[
\GHilb(M) \cong \FB_{e}(X)
\]
Our motivation is the following:

\begin{prob}\label{prob-1}
Understand a mechanism behind this phenomenon!
\end{prob}

The $G$-Hilbert scheme and the F-blowup 
fit into similar diagrams;
\begin{align*}
\xymatrix{
\text{Univ.\ fam.}\ar[d]\ar[r] &  M \ar[d] \\
 \GHilb(M) \ar[r] & M/G
}   \\
\intertext{and}
\xymatrix{
\text{Univ.\ fam.}\ar[d]\ar[r] &  X \ar[d] \\
 \FB_{e}(X) \ar[r] & X
} 
\end{align*}
In each diagram, the right vertical arrow is a finite and dominant morphism,
the left is finite and flat, the horizontal  ones are  projective and birational.
Moreover in each diagram,  the bottom one is the universal birational flattening of the right. 
A difference between the two diagrams is that 
the vertical arrows in the first are $G$-covers, in particular, separable,
while the ones in the second are purely inseparable. 
Bridgeland, King and Reid \cite{MR1824990} proved that
in the first diagram, under some condition, the Fourier-Mukai transform associated to the universal
family gives the equivalence 
\begin{equation}\label{isom-BKR}
\bD (\Coh(\GHilb(M))) \cong \bD(\GCoh(M)) .
\end{equation}
Here $\Coh (-)$ ($\GCoh(-)$) denotes the category of 
coherent ($G$-)sheaves
and $\bD(-)$ denotes the bounded derived category.
It is natural to ask:

\begin{prob}\label{prob-2}
Does a similar result hold for the second diagram?
\end{prob}

We will address Problems \ref{prob-1} and \ref{prob-2}
 in terms of the noncommutative resolution due to
Van den Bergh \cite{MR2077594,MR2057015}.
Now let us recall his observation.
For simplicity, suppose $M=\AA^{d}_{k}$ and $G \subset SL_{d}(k)$.
Let $S$ and $R$ be the coordinate rings of $M$ and $X=M/G$.
Then the endomorphism ring $A:=\End_{R}(S)$ is a noncommutative crepant resolution.
Namely $A$ is regular in the sense that it has finite global dimension
and satisfies the condition corresponding to the crepancy.
Now the $G$-Hilbert scheme is identified with some moduli
space $W$ of $A$-modules and a coherent $G$-sheaf on $M$
is identified with an $A$-module. 
Thus \eqref{isom-BKR} translates into 
\begin{equation}\label{translated-equivalence}
\bD (\Coh(W)) \cong \bD(A \lmodules) 
\end{equation}
and the Galois group $G$ disappears from view.
The equivalence in this form can fit to the situation of F-blowup.

Consider an affine scheme $X=\Spec R$ over $k$.
For $q =p^{e}$, the $e$-th Frobenius morphism of $X$
is identified with the morphism $X \to X_{e}:= \Spec R^{q}$
defined by the inclusion map $R^{q} \hookrightarrow R$.
We suppose that it is a finite morphism.
The relevant noncommutative ring is 
\[
D_{R,e}:=\End_{R^{q}}(R).
\]
Here particularly interesting is that $D_{R,e}$
is a ring of differential operators on $R$ and $\bigcup _{e} D_{R,e}$
is the ring of all differential operators on $R$.
The following is our answer to Problems \ref{prob-1} and \ref{prob-2}:

\begin{thm}
Let $M=\AA_{k}^{d}=\Spec S$, $G \subset GL _{d}(k)$ a small finite subgroup
and $X:= M/G = \Spec R$. Then for sufficiently large $e$, we have an equivalence of 
abelian categories 
\[
\End_{R}(S)\lmodules \cong D_{R,e} \lmodules.
\]
Hence $D_{R,e}$ has global dimension $d$.
Moreover $\GHilb(M)$ and $\FB_{e} (X)$ are the moduli spaces corresponds 
to each other via this equivalence, hence clearly isomorphic to each other.
\end{thm}

The point is that $R$ as well as $S$ and $S^q$ consists of 
the so-called modules of covariants as an $R^q$-module.
This was shown by Smith and Van den Bergh 
\cite{MR1444312} except for the fact that every module of covariants appears in $R$ (if $q$ is sufficiently large).
Then the last fact follows from Bryant's theorem \cite{MR1206130} in the representation theory 
(for details, see \S \ref{subsec-covariant}).

Now the equivalence \eqref{translated-equivalence} is directly translated 
to the F-blowup situation.
In view of these results, we may say that $D_{R,e}$
is a noncommutative counterpart of F-blowup.

\begin{rem}
It is natural that differential operators appear. 
For, in the Galois theory for purely inseparable extensions,
derivations play a role of automorphisms in the Galois theory of normal
extensions (see \cite{MR0392906}).
\end{rem}

We will also see that the F-blowup of an F-pure variety
can be expressed as the moduli space of $D_{R,e}$-modules
explicitly.

Since $D_{R,e}$ is defined for arbitrary $k$-algebra $R$,
we can ask:

\begin{prob}
When is $D_{R,e}$ a noncommutative (crepant) resolution?
\end{prob}

We see that for at least three different classes of singularities,
the answer is affirmative:

\begin{thm}
Suppose that $R$ is a complete local ring and has
one of the following singularity type:
\begin{enumerate}
\item 1-dimensional analytically irreducible singularity
\item tame quotient singularity
\item simple singularity of type $A_{1}$ (odd characteristic) 
\end{enumerate}
Then for sufficiently large $e$,  $D_{R,e}$ is a noncommutative resolution,
that is, of finite global dimension. (However it is not crepant in general. See Section 6.)
\end{thm}

The paper is organized as follows.
Sections 2 and 3 are devoted to the case of tame quotient singularities,
which is the core of the paper.  Here we prove the mentioned equivalence
$\End_{R}(S)\lmodules \cong D_{R,e} \lmodules$ and derive
the correspondence of moduli spaces. 
In Section 4, we  express an F-blowup of an F-pure variety
as the moduli space of $D_{R,e}$-modules explicitly. 
In Sections 5 and 6, we treat the 
1-dimensional analytically irreducible singularity
and the simple singularity of type $A_{1}$ respectively.

\subsection*{Acknowledgment}

This work was supported by Grant-in-Aid for Young Scientists (Start-up) 
(20840036, 20503882) and World Premier
International Research Center Initiative
(WPI Initiative), MEXT, Japan. The second author thanks Shoji Yokura 
for his help in reading  a paper written in German.

\subsection*{Convention and notation}

Throughout the paper, we work over an algebraically closed field $k$ of
characteristic $p > 0$. We always denote by $q$ the $e$-th
power of $p$ for some $e \in \Znn$.
For a commutative $k$-algebra $R$, we write $ R ^{q} := \{ f^{q} | f \in R \} \subset R$
and $D_{R,e}:= \End_{R^{q}}(R)$.
For the affine scheme $X=\Spec R$, we write $X_{e} := \Spec R^{q}$.
We will use the symbol, $\gldim$,
to denote the global dimension respectively.
For a ring $A$, we denote by $A\lmodules$ the category of left $A$-modules.

\section{Tame quotient singularities}

\subsection{Functors between module categories}

In this subsection, we collect a few results on
functors between module categories, which will be needed below.
These are probably well-known.

Fix a commutative ring $R$.
In the following, $L,M,N$ denote $R$-modules.
For convenience, we regard $L$ as a $(\End_{R}(L),R)$-bimodule,
and the same for $M,N$.
Then for instance, $\Hom_{R}(L,M)$ is a $(\End_{R}(M),\End_{R}(L))$-bimodule.

\begin{prop}\label{prop-compatible}
Suppose that $L$ is a direct summand of $M^{\oplus r}$ for some  $r \in \Zp$.
Then the natural morphism of $(\End_{R}(N),\End_{R}(L))$-bimodules
\begin{equation}\label{isom-bimodule}
\Hom_{R}(M,N) \otimes_{\End_{R}(M)}\Hom_{R} (L,M) \to \Hom_{R}(L,N) , \ f \otimes g \mapsto g \circ f
\end{equation}
is an isomorphism. Hence the composition of the functors
\[
 \Hom_{R} (L,M) \otimes_{\End_{R}(L)} - : \End_{R}(L) \lmodules \to \End_{R}(M)\lmodules
\]
and
\[
 \Hom_{R} (M,N) \otimes_{\End_{R}(M)} - : \End_{R}(M) \lmodules \to \End_{R}(N)\lmodules
\]
is canonically isomorphic to 
\[
 \Hom_{R} (L,N) \otimes_{\End_{R}(L)} - : \End_{R}(L) \lmodules \to \End_{R}(N)\lmodules.
\]
\end{prop}

\begin{proof}
By assumption, $\Hom_{R}(L,M)$ (resp.\ $\Hom_{R}(L,N)$)
is a direct summand of $\End_{R}(M)^{\oplus r}$ (resp.\ $\Hom_{R}(M,N)^{\oplus r}$). 
So \eqref{isom-bimodule} is a direct summand of the isomorphism
\[
\Hom_{R}(M,N) \otimes_{\End_{R}(M)}\End_{R}(M)^{\oplus r} \to \Hom_{R}(M,N)^{\oplus r}.
\]
It follows that \eqref{isom-bimodule} is also an isomorphism.
\end{proof}

The following is a direct consequence of the proposition.

\begin{cor}\label{cor-Morita}[cf. \cite[Corollary 2.5.8]{Quarles-thesis}]
Suppose that
for some positive integers $r$, $s$, $M$ is a direct summand
of $N^{\oplus s}$ and $N$ is a direct summand of $M^{\oplus r}$.
Then the functors 
\begin{align*}
\Hom_{R}(M,N) \otimes_{\End_{R}(M)} -:\End_{R}(M) \lmodules \to \End_{R}(N) \lmodules \\
\Hom_{R}(N,M) \otimes_{\End_{R}(N)} -:\End_{R}(N) \lmodules \to \End_{R}(M) \lmodules 
\end{align*}
are inverses to each other. In particular $\End_{R}(M)$ and $\End_{R}(N)$
are Morita equivalent.
\end{cor}

\begin{prop}\label{prop-idempotent}[cf. \cite[\S 19, Ex.\ 4]{MR1245487}]
Suppose that $M$ is a direct summand of $L$,
and $e \in \End_{R}(L) $ denotes the projection $L \twoheadrightarrow M \subset L$.
Then for any left $\End_{R}(L)$-module $N$, $e N $ is an $\End_{R}(M)$-module
and isomorphic to $\Hom_{R}(L,M) \otimes_{\End_{R}(L)} N$.
\end{prop}

\begin{proof}
Since the functors $N \mapsto e N$ and $N \mapsto \Hom_{R}(L,M) \otimes_{\End_{R}(L)}N$
are exact, it suffices to show the proposition in the case $N = \End_{R}(L)$,
which we can see as follows. Write $L = M \oplus M'$. 
Then 
\[
\End_{R}(L) =
\begin{pmatrix}
\End_{R}(M) & \Hom_{R} (M',M) \\
\Hom_{R}( M,M' )& \End_{R}(M')
\end{pmatrix}
\text{ and } 
e= \begin{pmatrix}
1 & 0\\
0& 0
\end{pmatrix}.
\]
Hence 
\[
e \End_{R}(L)=
\begin{pmatrix}
\End_{R}(M) & \Hom_{R} (M',M) \\
0&0
\end{pmatrix}
= \Hom_{R}(L,M) . 
\]
\end{proof}

\subsection{Modules of covariants and Frobenius maps}\label{subsec-covariant}

Let $V$ be a $d$-dimensional $k$-vector space
and  $G \subset GL (V)$ a finite subgroup.
We assume the tameness condition that  $p$ does not divide $\sharp G$.
Let $S$ be the symmetric algebra $S^{\bullet} V$ with the natural $G$-action.
Set  $R:= S^{G}$, the invariant ring.

For a finite dimensional 
$G$-representation $U$, $R(U):=(S \otimes_{k} U)^{G}$ is a finitely generated
$R$-module,
called a \emph{module of covariants} (over $R$).
Let  $U_{1},\dots,U_{l}$ be the complete set of irreducible representations.
Then every module of covariants is the direct sum of copies of $R(U_{i})$'s.
We say that $R(U)$ is \emph{full} if it contains every $R(U_{i})$ as a direct summand.
Since $S\cong (k[G] \otimes_{k} S)^{G}$, $S$ is a full module of covariants.
Similarly $S^{q}$ is a full module of covariants over $R^{q}$.

\begin{lem}\label{lem-Bryant}
Let $\fm:= S^{>0}V \subset S $ be the homogeneous maximal ideal
and $\fm^{[q]}$ its $e$-th Frobenius power, that is, 
the ideal of $S$ generated by $f^{q}$, $f \in \fm$.
\begin{enumerate}
\item\label{Bryant-1}
For sufficiently large $q$, the quotient $G$-representation
$S/\fm^{[q]}$ of $S$
contains all irreducible $G$-representations as direct summands.
\item In addition if $G$ is abelian, then the preceding assertion
holds for  $q \ge \sharp G $.
\end{enumerate}
\end{lem}

\begin{proof}
\begin{enumerate}
\item
It follows from Bryant's theorem \cite{MR1206130} that
for large $l$, 
the set of polynomials of degree at most $l$, $S^{\le l} V \subset S$,
contains all irreducible representations. 
Then if $q \gg l$, 
since  $ \fm^{[q]} \subset S^{>l} V $, 
the natural map $S^{\le l} V \to S/\fm^{[q]}$ 
is injective,
and the assertion follows.
\item 
There is a decomposition of $V$ into 1-dimensional representations,
\[
V =  V_{1} \oplus \dots \oplus V_{d}.
\]
Again from Bryant's theorem, every irreducible representation
is of the form $V_{1}^{\otimes n_{1}} \otimes \cdots \otimes V_{d}^{n_{d}}$ for some $n =(n_{1},\dots,n_{d})$,
$0 \le  n_{i} < \sharp G $. Now it is easy to see the assertion. 
\end{enumerate}
\end{proof}

\begin{prop}\label{prop-full-frobenius}
Suppose that $q$ is large enough as in the preceding lemma.
(In particular, if $G$ is abelian, $q \ge \sharp G$ is enough.)
Then $R$ is a full module of covariants over $R^{q}$.
\end{prop}

\begin{proof}
Note that the proposition is a direct consequence of \cite[Proposition 3.2.1]{MR1444312}
and the preceding lemma. Indeed since $S$ is isomorphic to 
$S/\fm^{[q]} \otimes_{k} S^{q}$, $R$ is isomorphic to $(S/\fm^{[q]} \otimes_{k} S^{q})^{G}$.
Therefore the proposition follows from the lemma.
\end{proof}

\begin{rem}
If in the non-abelian case, 
we had an effective estimation on how large $q$ is enough in 
Lemma \ref{lem-Bryant}, 
then we would have one in Proposition \ref{prop-full-frobenius} too.
In characteristic zero, the assertion of Lemma \ref{lem-Bryant} 
is valid under the condition $q \ge \sharp G$ even in the non-abelian case 
(see \cite[Prop.\ II.1.3 and Cor.\ II.3.4]{Panyushev-lecture-representation-finite-invariant}).
\end{rem}

\begin{cor}
If $q$ is large enough as above, 
the functors
\[
\Phi:= \Hom_{R^{q}}(S^{q},R) \otimes_{\End_{R^{q}}(S^{q})} - : \End_{R^{q}}(S^{q})\lmodules
\to D_{R,e}\lmodules 
\]
and
\[
 \Hom_{R^{q}}(R, S^{q}) \otimes_{D_{R,e}} - :D_{R,e}\lmodules
\to \End_{R^{q}}(S^{q})\lmodules 
\]
are equivalences which are inverses to each other.
\end{cor}

\begin{proof}
This follows from Corollary \ref{cor-Morita} and Proposition \ref{prop-full-frobenius}.
\end{proof}

\begin{cor}
If $q$ is large enough as above, then $ \gldim D_{R,e}= d$.
Furthermore $D_{R,e}$ is a  Cohen-Macaulay $R^{q}$-module. 
The same is true for the completion of $R$ with respect to the maximal ideal $\fm \cap R$.
\end{cor}

\begin{proof}
We may and will suppose that $d \ge 2$ and 
that $G$ is small, that is,  has no reflection.
Then $ \End_{R^{q}}(S^{q}) $ is isomorphic to 
the skew group ring $S^{q}[G]$ (see \cite{MR0137733}).
It is well-know that $S^{q}[G]$ has global dimension $d$.
Being Morita equivalent to $S^{q}[G]$, $D_{R,e}$ also has global dimension $d$. 

We easily see that $S^{q}[G]$
is a Cohen-Macaulay $R^{q}$-module.
Since 
\begin{align*}
S^{q}[G] \cong \End_{R^{q}}(S^{q})  \cong \bigoplus_{i,j} \Hom_{R^{q}}(R^{q}(U_{i}),R^{q}(U_{j}))^{a_{ij}}, \ a_{ij} >0,  \\
D_{R,e} \cong \bigoplus_{i,j} \Hom_{R^{q}}(R^{q}(U_{i}),R^{q}(U_{j}))^{b_{ij}}, \ b_{ij} >0, 
\end{align*}
$D_{R,e}$ is also Cohen-Macaulay. 
The last assertion is obvious. 
\end{proof}

\begin{rem}
The Cohen-Macaulayness is the condition corresponding the crepancy
in Van den Bergh's sense (see \cite[Lemma 4.2]{MR2077594}). 
Note however that he defined the noncommutative crepant resolution
only for Gorenstein singularities.
\end{rem}

\section{Moduli spaces of stable objects}

We keep the notation of the preceding section.
Suppose that $q$ is sufficiently large. 

\subsection{Stability}

For $i =1,\dots,l$, define functors
\begin{align*}
\Psi_{1,i}:= \Hom_{R^{q}}(S^{q},R^{q}(U_{i})) \otimes _{\End_{R^{q}}(S^{q})}- : \End_{R^{q}}(S^{q}) \lmodules \to \End _{R^{q}} (R^{q}(U_{i})) \lmodules \\
\Psi_{2,i} : \Hom_{R^{q}}(R,R^{q}(U_{i})) \otimes _{D_{R,e}}- : D_{R,e} \lmodules \to  \End _{R^{q}} (R^{q}(U_{i})) \lmodules .
\end{align*}
From Proposition \ref{prop-compatible}, for each $i$, $\Psi_{1,i}$ and $\Psi_{2,i}$ are compatible with
the equivalence $\Phi$.

\begin{defn}
For a (closed) point $x$ of $X_{e}:=\Spec R^{q}$
and a left $\End_{R^{q}}(S^{q}) \otimes_{R^{q}} k(x)$-module $V$, 
we define the \emph{dimension vector} $\dimbf V := (\dim_{k} \Psi_{1,i}(V) )_{i} \in \ZZ^{l}$.

Let $\lambda \in \ZZ^{l}$.  We say that $V$ is \emph{$\lambda$-(semi)stable} 
 if $(\lambda, \dimbf V)=0$
and for any proper submodule $W$ of $V$, we have $(\lambda, \dimbf W) >0$ ($\ge 0$).
Here $(-,-)$ is the standard inner product. 
We similarly define the dimension vector and the (semi)stability for $D_{R,e} \otimes_{R^{q}} k(x)$-modules
with $\Psi_{2,i}$ instead of $\Psi_{1,i}$.

Given $\alpha \in \Znn^{l}$, we say that $\lambda$ is
\emph{generic with respect to dimension vector $\alpha$}
if every $\lambda$-semistable $V$ with $\dimbf V =\alpha$
is $\lambda$-stable.
\end{defn}

Our stability corresponds to Craw-Ishii's one \cite{MR2078369} as follows.

\begin{defn}
A \emph{$G$-constellation} on $\Spec S^{q}(\cong \AA_{k}^{d})$ 
is an $S^{q}[G]$-module $M$ which is isomorphic to $k[G]$
as a $k[G]$-module. A \emph{$G$-cluster} is a $G$-constellation which is
a quotient of $S^{q}$.
Let $\mathbf{R} (G):= \bigoplus _{i} \ZZ [U_{i}]$ be the representation ring
and 
\[
\theta : \mathbf{R}(G) \to \ZZ
\]
 a map of abelian groups  such that $\theta (k[G]) =0$.
A $G$-constellation $M$ is \emph{$\theta$-stable} (resp.\ \emph{$\theta$-semistable})
if for every proper $S[G]$-submodule $L \subset M$,
$\theta(L) >0$ (resp.\ $\ge 0$) . 
\end{defn}

For such $\theta$, put $\lambda_{\theta} := (\theta(U_{1}^{*}), \dots, \theta(U_{l}^{*}))$. 
Here $U_{i}^{*}$ is the dual representation of $U_{i}$.
A $G$-constellation $M$ naturally becomes an $\End_{R^{q}}(S^{q}) \otimes _{R^{q}} k(x)$-module
for some and unique $x \in X_{e} $.
From the following lemma, $\theta(M) =(\lambda_{\theta}, \dimbf M)$.
Hence
the $\theta$-(semi)stability of the $G$-constellation in the sense of Craw-Ishii
coincides with the $\lambda_{\theta}$-(semi)stability of the $\End_{R^{q}}(S^{q}) \otimes _{R^{q}} k(x)$-module.

\begin{lem}
Let $N$ be an $\End_{R^{q}}(S^{q})$-module and 
$N=\bigoplus_{I=1}^{l} N_{U_{i}}$
the isotypic decomposition of $N$ as a $G$-representation.
(That is, for each $i$, $N_{U_{i}}$ is a direct sum of copies of $U_{i}$.)
Then for each $i$, we have a natural isomorphism 
\[
\Hom_{R^{q}} (S^{q},R^{q}(U_{i})) \otimes_{\End_{R^{q}}(S^{q})} N \cong  (N _{U_{i}^{*}} \otimes U_{i})^{G}.
\]
In particular, if $ N_{U_{i}^{*}} =(U_{i}^{*})^{\oplus r}$ $(r  < \infty)$, then 
$\dim_{k} \Psi_{1,i}(N) =r $.
\end{lem}

\begin{proof}
We first see
\begin{align*}
\Hom_{R^{q}}(S^{q},R^{q}(U_{i})) & = \Hom_{R^{q}}(S^{q},(S^{q} \otimes U_{i})^{G}) \\
& \cong ( \Hom_{R^{q}}(S^{q},S^{q}) \otimes U_{i})^{G} \\
& \cong (\End_{R^{q}}(S^{q})_{U_{i}^{*}} \otimes U_{i})^{G}.
\end{align*}
Thus the first assertion holds when $N$ is a free module. 
We can prove the general case by using a free presentation of $N$.

The first assertion implies the second. 
\end{proof}

Our stability also corresponds to Van den Bergh's one \cite{MR2077594}
as follows. Write $S^{q}= \bigoplus_{i} R^{q}(U_{i}) ^{\oplus r_{i}} $ and 
$e_{i} :S^{q} \to R^{q}(U_{i})^{\oplus r_{i}} \subset R$
the projections. Then the $e_{i} \in \End_{R^{q}}(S^{q})$ are pairwise orthogonal idempotents
with $\sum _{i} e_{i} =1$. From Proposition \ref{prop-idempotent}, 
for an $\End_{R^{q}}(S^{q})$-module $N$, 
we have 
\[
\Psi_{1,i}(N)^{\oplus r_{i}} \cong 
\Hom_{R^{q}}(S^{q},R^{q}(U_{i})^{\oplus r_{i}}) \otimes_{\End_{R^{q}}(S^{q})} N
\cong e_{i} N.
\] 
Hence if write $r = (r_{1},\dots,r_{l})$, 
 the dimension vector $\dimbf' N$ of  $N$ with respect to the $e_{i}$
in the sense of Van den Bergh is equal to $ r \cdot \dimbf N $ (componentwise multiplication). 
Consequently, for $\lambda \in \ZZ^{l}$, putting $\lambda' := (\lambda_{1}/r_{1} ,\dots \lambda_{l}/r_{l})$,
we have $(\lambda' , \dimbf' N) = (\lambda, \dimbf N)$.
It follows that our stability with respect to $\lambda$
corresponds to Van den Bergh's one with respect to the $e_{i}$ and $\lambda'$.
Similarly for $D_{R,e}$ in place of $\End_{R^{q}}(S^{q})$.

\begin{rem}
The stabilities of Craw-Ishii and Van den Bergh are both derived from King's \cite{MR1315461}
and hence must a priori correspond to each other. 
We just described their stabilities with functors and 
confirmed their compatibility with the equivalence of module categories. 
\end{rem}

\subsection{Moduli space}

Choose $\alpha \in \Znn ^{l}$ and $\lambda \in \ZZ^{l}$
which is generic with respect to $\alpha$.
Let $A$ be either $\End_{R^{q}}(S^{q})$ or $D_{R,e}$.
From \cite{MR2077594}, there exists the fine
moduli space $W$ of $\lambda$-stable $A$-modules
 with dimension vector $\alpha$. 
This is a projective $X_{e}$-scheme, say with the structure morphism
$\pi :W \to X_{e}$. 
Each point $x \in W$ represents 
an isomorphism class of $\lambda$-stable $A\otimes_{R^{q}} k(\pi(x))$-modules with dimension vector $\alpha$.

When $A = \End_{R^{q}}(S^{q})\cong S^{q}[G]$ and when $\alpha = \dimbf k[G]$
and $\lambda=\lambda_{\theta}$ as above, 
then $W$ is the moduli space of $\theta$-stable $G$-constellations.
If $\theta (U_{i}) >0$
for all nontrivial representations $U_{i}$, then $W$ is the moduli space of
the $G$-clusters (see \cite[page 266]{MR2078369} and also \cite{MR1783852}), that is, the $G$-Hilbert scheme.

If we replace $\End_{R^{q}}(S^{q})$ with
$D_{R,e}$ and keep $\alpha$ and $\lambda$ unchanged,
then the resulting moduli space $W'$
is canonically isomorphic to $W$.  For  $W$ and $W'$ are the moduli spaces of stable objects of
two equivalent abelian categories respectively with respect to stability conditions
corresponding to each other. 
Thus the canonical isomorphism $W \to W'$ is nothing but the restriction 
of the equivalence $\Phi$ to objects belonging to $W$.

\subsection{The maps from the $G$-Hilbert scheme to F-blowups}

From Proposition \ref{prop-compatible}, the equivalence $\Phi: \End_{R^{q}}(S^{q})\lmodules \to
D_{R,e} \lmodules$ factors into two functors
\begin{align*}
\Phi_{1} := \Hom_{R^{q}}(S^{q},S) \otimes_{\End_{R^{q}}(S^{q})}- : \End_{R^{q}}(S^{q})\lmodules
\to \End_{R^{q}}(S)\lmodules \\
\Phi_{2} : = \Hom_{R^{q}}(S,R) \otimes_{\End_{R^{q}}(S)}- : \End_{R^{q}}(S)\lmodules \to
D_{R,e}\lmodules .
\end{align*}
For a left $\End_{R^{q}}(S^{q})$-module $M$, we have  isomorphisms of $S$-modules,
\[
\Phi_{1}(M) \cong (\End_{R^{q}}(S^{q}) \otimes_{S^{q}}S) \otimes_{\End_{R^{q}}(S^{q})} M \cong  S \otimes_{S^{q}} M .
\]
Thus if we forget the $\End_{R^{q}}(S)$-module structure
and remember only the $S$-module structure, 
then  $\Phi_{1}$ is just the pull-back by the Frobenius morphism
$\Spec S \to \Spec S^{q}$, which is flat of rank $q^{d}$ since $\Spec S$ is smooth. 

On the other hand, for a left $\End_{R^{q}}(S)$-module $N$,
the group $G$, regarded as a subset of $ \End_{R^{q}}(S)$, acts on
$N$ and we have a natural isomorphism of $R$-modules
\[
\Phi_{2}(N) \cong N^{G}.
\]
Indeed the isomorphism is obvious if $N$ is free.
In the general case, we can show this, considering a free presentation of $N$.
In summary,  as $R$-modules, $\Phi(M) \cong (S \otimes_{S^{q}}M)^{G}$.

If  $\theta :\mathbf{R}(G) \to \ZZ$  is such that
$\theta(U_{i}) >0$ for every nontrivial irreducible representation $U_{i}$,
then as mentioned above,
 the corresponding moduli space $W$ of $G$-constellation is the $G$-Hilbert scheme.
Now the isomorphism $\Phi : W \to W'$ restricted to the irreducible
component of $W$ dominating $X_{e}$
coincides with the isomorphism of the $G$-Hilbert scheme
to the $F$-blowup constructed in \cite{Yasuda:math0706.2700}.
In particular, the $e$-th F-blowup of $\Spec R$ is
an irreducible component of the moduli space $W$
of stable $D_{R,e}$-modules for the stability $\lambda_{\theta}$.

Summarizing, we have:

\begin{thm}
The $e$-th F-blowup of $X=\Spec R$ is (an irreducible component of) the moduli space of certain 
$D_{R,e}$-modules which corresponds to the $G$-Hilbert scheme
via the equivalence $\Phi$.
Hence the isomorphism between the $G$-Hilbert scheme and the F-blowup
is just the restriction of the equivalence. 
\end{thm}

\subsection{Fourier-Mukai transform}

Applying \cite[Theorem 6.3.1]{MR2077594} to our situation,
we obtain: 

\begin{cor}
(We still suppose that $e$ is sufficiently large.)
Let $Y \to X_{e}$ be the $e$-th F-blowup of $X$. 
Suppose that $\dim Y \times_{X} Y \le d+1$ and that $G \subset SL_{d}(k)$.
Then $Y$ is a crepant resolution and we have the equivalence 
\[
\bD(\Coh(Y)) \cong \bD(D_{R,e}\lmodules) 
\]
defined as the Fourier-Mukai transform associated to
the the universal
family of $D_{R,e}$-modules over $Y$.
\end{cor}

\section{F-blowups of F-pure singularities}

Let $R$ be a commutative finitely generated domain over $k$ of dimension $d$.
Suppose that $R$ is F-pure, that is, the inclusion map $R^{p} \hookrightarrow R$ splits
(as an $R^{p}$-homomorphism).
Then for any $q=p^{e}$, $R^{q} \hookrightarrow R$ splits.
So we write $R = R^{q} \oplus M $ with $M \subset R$ an $R^{q}$-submodule.
Let $e_{1},e_{2} \in \End_{R^{q}}(R)$ be the projections 
$e_{1}:R \to R^{q} \subset R$ and $e_{2}: R \to M \subset R$ respectively.
Then we can consider dimension vectors and the stability
with respect to $e_{1}$ and $e_{2}$ in Van den Bergh's sense. 

\begin{prop}
Put $\alpha := (1,q^{d}-1)$ and $\lambda:= (1-q^{d},1)$.
Then the $e$-th F-blowup of $X$ is
canonically isomorphic to the unique irreducible component
dominating $X_{e}$ of 
the moduli scheme $W$ of $\lambda$-stable $D_{R,e}$-modules
with dimension vector $\alpha$.
\end{prop}

\begin{proof}
First note that $\lambda$ is generic with respect to $\alpha$,
since for any $0 < \alpha' < \alpha $, $(\alpha',\lambda)\ne 0$.
The proof here is a modification of Craw-Ishii's argument \cite[page 266]{MR2078369}.
If $x \in X_{e}$ is a smooth point, then $D_{R,e}\otimes_{R^{q}} k(x) \cong 
M_{q^{d} \times q^{d}}(k)$ and there exists one and only one $D_{R,e} \otimes k(x)$-module
of $k$-dimension $q^{d}$ modulo isomorphisms.
A canonical representative of it is $R \otimes _{R^{q}} k(x)$ with the canonical
$D_{R,e} \otimes k(x)$-module structure. 
It has dimension vector $(1,q^{d}-1)$ with respect to the idempotents
$e_{1},e_{2}$. It follows that $W$ has a unique irreducible component
dominating $X_{e}$.

If $x \in X_{e}$ is an arbitrary point, then any $D_{R,e} \otimes k(x)$-module
$M$ of dimension vector $\alpha$
which is a quotient of $R \otimes _{R^{q}} k(x)$ is $\lambda$-stable.
Indeed if $M' \subset M$ is a submodule with $e_{1}M' \ne 0$,
then we must have $M' = M$. 
Hence for every nonzero proper submodule $L \subset M$, we have
$\dimbf L = (0, r)$, $r >0$ and $(\lambda, \dimbf L) =r >0$. 
Thus $M$ is $\lambda$-stable.
It follows that the $e$-th F-blowup is a subscheme of $W$,
which completes the proof.
\end{proof}

\begin{rem}
The F-blowup of F-pure singularities has another nice property. Namely 
the sequence of F-blowups satisfies the monotonicity \cite{yasuda-monotonicity}.
\end{rem}

\section{1-dimensional analytically irreducible singularities}

Let $R$ be a 1-dimensional complete integral domain over $k$
with the normalization $S :=k[[x]]$. 

\begin{lem}
For any $q$, we have ring isomorphisms
\[
\End _{R^{q}} (S^{q}) \cong \End_{S^{q}} (S^{q}) \cong S^{q}.
\]
\end{lem}

\begin{proof}
The second isomorphism is trivial. For the first one, 
we first see 
\[
\End_{R^{q}} (S^{q}) \supset \End_{S^{q}}(S^{q}).
\]
Take $ \phi \in \End_{R^{q}} (S^{q})$. For any $ s \in S^{q}$,
there exists $n \in q \Znn$ with $x^{n},x^{n}s \in R^{q}$.
Then
\[
x^{n} \phi (s) = \phi(x^{n}s) =  x^{n}s\phi(1) .
\]
Hence $\phi(s) = s \phi(1)$ and $\phi \in \End_{S^{q}}(S^{q})$.
Thus we have proved the lemma.
\end{proof}

\begin{thm}
Suppose that $x^{q} \in R$. (This holds for sufficiently large $q$.)
Then we have the ring isomorphism
\[
D_{R,e} \cong M_{q \times q}(S^{q}).
\]
\end{thm}

\begin{proof}
By assumption, $S^{q} \subset R$. 
Being a torsion-free $S^{q}$-module of rank $q$,
$R$ is isomorphic to $(S^{q})^{\oplus q}$
as an $S^{q}$-module and as an  $R^{q}$-module too.
It follows that 
\[
D_{R,e}\cong \End_{R^{q}}((S^{q})^{\oplus q}) \cong M_{q \times q} (\End_{R^{q}}(S^{q})) \cong M_{q \times q} (S^{q}) .
\]
\end{proof}

\begin{cor}
For sufficiently large $e$, $\gldim D_{R,e} = 1$.
Furthermore $D_{R,e}$ is a Cohen-Macaulay $R^{q}$-module.
\end{cor}

\begin{proof}
It is well-known that $M_{q \times q} (S^{q})$ is Morita equivalent to $S$.
Hence $D_{R,e}$ has global dimension 1.

It is clear that $D_{R,e} \cong  M_{q \times q}(S^{q})$ is a Cohen-Macaulay
$R^{q}$-module. 
\end{proof}

\begin{prob}
For large $e$, the $e$-th $F$-blowup of a curve
separates all the analytic branches \cite{Yasuda:math0706.2700}.
Is there a categorical counterpart of separation of branches?
And can we generalize the result to arbitrary 1-dimensional singularity?
\end{prob}

See the next section for the node $R=k[[x,y]]/(x^{2}+y^{2})$
in odd characteristic.

\section{The simple singularity of type $A_{1}$}

A hypersurface singularity $R=k[[x_{0}, \dots ,x_{d}]]/(f)$, 
which is necessarily a reduced ring, is called \emph{simple},
if it is of finite Cohen-Macaulay type, that is, 
it has only finitely many indecomposable MCM (maximal Cohen-Macaulay) modules
up to isomorphisms. See \cite{MR1033443} for the classification of simple singularities
in positive characteristic.

For such $R$,  an MCM $R$-module $M$ is called a \emph{representation generator}
if $M$ contains every indecomposable MCM module as a direct summand. 
The following theorem of Leuschke \cite{MR2310620} provides a useful
sufficient condition for an endomorphism ring to have finite global dimension.

\begin{thm}
For $R$ as above, if $M$ is a representation generator,
then 
\[
\gldim \End_{R}(M) \le \max\{2,d\}.
\] 
Moreover if $d \ge 2$, then the equality holds. 
\end{thm}

Now suppose that $k$ has odd characteristic
and $R$ is the $d$-dimensional simple singularity of type $A_{1}$;
\[
R= k[[x_{0}, \dots,x_{d}]] / (x_{0}^{2} + x_{1}^{2} + \dots +x_{d}^{2}) .
\]
Then we see as follows that for $e > 1$, 
\[
\gldim D_{R,e} \le \max\{2,d\} 
\] 
with the equality in the case $d \ge 2$.

\begin{proof}
Since the monomial $x_{0}^{p-1}x_{1}^{p-1}$ with nonzero coefficient
$\binom{p-1}{(p-1)/2}$ appears in 
$(x_{0}^{2} + x_{1}^{2} + \dots +x_{d}^{2})^{p-1}$,
we have $(x_{0}^{2} + x_{1}^{2} + \dots +x_{d}^{2})^{p-1} \notin (x_{0}^{p}, \dots, x_{d}^{p})$.
From Fedder's criterion \cite[Proposition 2.1]{MR701505}, $R$ is F-pure,
that is, for any $q$, $R^{q} \hookrightarrow R$ splits as an $R^{q}$-module homomorphism.

If $d$ is even, then $R$ has only two indecomposable MCM modules, one of which
is the trivial one. To see this, it is enough to check the case where $d=2$,
thanks to the Kn\"orrer periodicity \cite{MR877010}.
In this case, it was proved by Auslander \cite{MR816307}.
The same is true for $R^{q}$.
 From Kunz \cite{MR0252389}, $R$ is not a free $R^{q}$-module.
So $R$ must be a representation generator as an $R^{q}$-module
and the assertion follows.

If $d=1$, $R$ has three indecomposable MCM modules,
$R$, $R/(x_{0}+x_{1})$ and $R/(x_{0}-x_{1})$.
The two nontrivial MCM modules are interchanged by  a coordinate change. 
In fact, these three are the only indecomposable MCM modules
 (see Dieterich and Wiedemann \cite{MR825715}  
and Kiyek and Steinke \cite{MR818299}).
Then the same holds for $R^{q}$.
Again from the Kn\"orrer periodicity, the same holds for arbitrary odd $d$.
Since $R$ is not a free $R^{q}$-module, $R$ contains one of the two nontrivial
indecomposable MCM modules as a direct summand. 
But from the symmetry, it must contain the other too and hence is a representation
generator. The assertion follows. 
\end{proof}

If $d \ge 3$, then from a result of Quarles \cite[Corollary 3.5.5]{Quarles-thesis},
 the above $D_{R,e}$ is not Cohen-Macaulay. Namely it is not  crepant
 in the sense of Van den Bergh \cite{MR2077594}.  

\begin{prob}
How about simple singularities of other types?
\end{prob}

\bibliographystyle{hsiam}
\bibliography{mybib}

\end{document}